\documentclass[11pt]{article}
\usepackage{amsmath} 
\usepackage{amsthm}
\usepackage{amssymb} 

\usepackage{comment} 
\usepackage[english]{babel} 
\theoremstyle{plain}
\newtheorem{thm}{Theorem}
\newtheorem*{thm*}{Theorem}

\newtheorem*{cla*}{Claim}
\newtheorem*{lem*}{Lemma}
\newtheorem{prop}[thm]{Proposition}
\theoremstyle{remark}
\newtheorem{rmk}{Remark}
\numberwithin{equation}{section}

 \def\XXint#1#2#3{{\setbox0=\hbox{$#1{#2#3}{\int}$}
 \vcenter{\hbox{$#2#3$}}\kern-.5\wd0}}

 \usepackage{hyperref}

\begin{document}
\author{Susanna Risa  \footnote{Dipartimento di Matematica, Universit\`a di Roma ``Tor Vergata'', Via della Ricerca Scientifica 1, 00133, Roma, Italy. E-mail: risa@mat.uniroma2.it}
, Carlo Sinestrari \footnote{Dipartimento di Ingegneria Civile e Ingegneria Informatica, Universit\`a di Roma ``Tor Vergata'', Via Politecnico 1, 00133, Roma, Italy. E-mail: sinestra@mat.uniroma2.it}
}
\title{Ancient solutions of geometric flows with curvature pinching}
\date{}
\maketitle
\begin{abstract}
We prove rigidity theorems for ancient solutions of geometric flows of immersed submanifolds. Specifically, we find pinching conditions on the second fundamental form that characterise the shrinking sphere among compact ancient solutions for the mean curvature flow in codimension greater than one, and for some nonlinear curvature flows of hypersurfaces.
\end{abstract}\medskip

%\noindent {\bf MSC 2010 subject classification} 53C44, 35B40 \medskip

\section{Introduction}

In this paper, we study ancient solutions of the mean curvature flow and other geometric evolutions of immersed submanifolds. We recall that a solution is called {\em ancient} if it is defined on a time interval of the form $(-\infty,T)$.  These solutions typically arise as limits of rescalings and model the asymptotic profile of the flow near a singularity, see e.g \cite{H2}. They have also been considered in theoretical physics, where they appear as steady state solutions of the renormalization-group flow in the boundary sigma model \cite{BS, LVZ}.

In recent years, various researchers have investigated ancient solutions of the mean curvature flow in codimension one. In particular, attention has been focused on ancient solutions which are convex, since this property is enjoyed by the blowup limits of general mean convex solutions, see \cite{HS99,W03}. The easiest example is provided by the shrinking sphere, which is the only compact convex homothetically contracting solution. An example of an ancient solution which is not self-similar is the so-called {\em Angenent oval} \cite{An1}. It is a convex solution of curve shortening flow in the plane which has larger and larger eccentricity as $t \to -\infty$.  Compact convex ancient solutions in the plane have been completely classified by Daskalopoulos-Hamilton-Sesum \cite{DHS} to be either shrinking round circles or Angenent ovals.

In higher dimensions, other examples of non-homothetical ancient solutions of the mean curvature flow have been constructed and analysed in \cite{Ang, ADS, BLT, HH, W03}. These examples suggest that the class of convex compact ancient solutions of the mean curvature flow is wide and that a complete classification is difficult to obtain. However, it is possible to prove some rigidity results which provide a partial structural description of this class. The typical result of this kind is the following, proved in \cite{HS}, see also \cite{HH}, stating that an ancient solution with uniformly pinched principal curvatures is necessarily a shrinking sphere.

\begin{thm}\label{Hui-Sin}
Let $\{M_t\}_{t \in (-\infty,0)}$ be a family of closed convex hypersurfaces of $\mathbb{R}^{n+1}$ evolving by mean curvature flow, with $n>1$. Let $\lambda_1$ denote the smallest principal curvature of $M_t$ and $H$ the mean curvature. Suppose that there exists $\epsilon>0$ such that $\lambda_1 \geq \epsilon H$ on $M_t$ for every $t$. Then $M_t$ is a family of shrinking spheres.
\end{thm}

Several other characterizations of the sphere have been obtained, e.g. in terms of a control on the diameter growth as $t \to -\infty$, or on the ratio between the outer and inner radius, see \cite{HH,HS}. We remark that an equivalence similar to Theorem \ref{Hui-Sin}, involving pinching properties of the intrinsic curvature, has been obtained in \cite{BHS} for ancient solutions of the Ricci flow. Another interesting rigidity result, proved by Haslhofer and Kleiner \cite{HK}, is that a closed mean convex ancient solution which is uniformly noncollapsed in the sense of Andrews \cite{A4} is necessarily convex. A similar statement has been obtained by Langford \cite{L1} replacing the noncollapsing property by the assumption that $\lambda_1/H$ is bounded from below. After this, Langford and Lynch \cite{LL} have generalized the above results to a large class of flows with speed given by a function of the curvatures homogeneous of degree one. Very recently, Bourni, Langford and Tinaglia \cite{BLT} have showed uniqueness of rotationally symmetric collpased ancient solutions in all dimensions.

Various authors have also investigated the case of the mean curvature flow in the sphere $\mathbb{S}^{n+1}$. In this case, there is an ancient solution consisting of shrinking geodesic $n$-dimensional spheres, converging to an equator as $t \to -\infty$ and to a point as $t \to 0$. In \cite{HS}, some characterizations of the shrinking round solution in terms of curvature pinching have been given. Compared to the case of Euclidean ambient space, it can be observed that the positive curvature of the ambient space increases rigidity and weaker pinching conditions suffice to obtain the result. 
A strong result in this context has been obtained by Bryan, Ivaki and Scheuer \cite{BIS}, who considered curvature flows on the sphere $\mathbb{S}^{n+1}$ for very general speed functions, and proved that the shrinking geodesic sphere is the only closed convex ancient solution with bounded curvature for large negative times. In dimension $1$ (i.e. for Curve Shortening Flow), Bryan and Louie \cite{BrLo} have proved that embeddedness is sufficient to conclude that a closed ancient curve in $\mathbb{S}^2$ is either a shrinking circle or an equator.

The aim of this paper is to derive rigidity results for ancient solutions of more general curvature flows, by considering either mean curvature flow in higher codimension, or the hypersurface case with more general speeds than mean curvature.
We consider various kinds of flows, but all our results are in a similar spirit to Theorem \ref{Hui-Sin}, showing that a suitable uniform pinching condition characterizes the shrinking sphere among convex ancient solutions. The choice of the pinching condition is inspired in many cases by the previous works in the literature that have analysed the formation of singularities in finite time of convex solutions. We give here an outline of the paper and summarize our main results.

In Section 2 we consider mean curvature flow of general codimension. After some preliminaries, we consider in \S 2.1 the flow in Euclidean space
\begin{equation} \label{MCFH}
\frac{\partial \varphi}{\partial t} (x,t) = \textbf{H}(x,t),
\end{equation}
where $\varphi (\cdot,t)$ is a family of immersions of a closed $n$-dimensional manifold $M^n$ into $\mathbb{R}^{n+k}$ and $\textbf{H}$ is the mean curvature vector. Here both $n$ and $k$ are at least $2$. We consider ancient solutions satisfying a pinching condition of the form $|h|^2 \leq C |\textbf{H}|^2$, where $|h|$ is the norm of the second fundamental form and $C$ is a suitable constant depending on $n$ (see \eqref{pinch}). The same condition was considered by Andrews and Baker \cite{AB} in the study of finite-time singularities; recently, a similar result has been obtained by Pipoli and Sinestrari for submanifolds of the Complex Projective Space \cite{pipolisin}. We adapt the method of \cite{HS} by using the estimates of Andrews and Baker, in addition we follow a technique by Hamilton \cite{H2} to obtain a bound on the intrinsic diameter in general codimension. In this way, we prove that the only ancient solutions satisfying the pinching condition are the $n$-dimensional shrinking spheres. In \S 2.2 we obtain analogous results for the same flow when the manifolds are immersed in the sphere. As in the previous works in codimension one, positive curvature of the ambient space increases rigidity and a weaker pinching condition suffices to characterize the shrinking round solutions.

Section 3 is devoted to flows with nonlinear speed in codimension one. The evolution has the form
\begin{equation}\label{flowh}
\frac{\partial \varphi}{\partial t}(x,t) = -F(W(x,t))\nu(x,t) = -f(\lambda(W(x,t)))\nu(x,t),
\end{equation}
where the speed function $f$ is a positive symmetric homogeneous function of the eigenvalues of the Weingarten operator $W$, with some additional properties to be specified later. Again, under these flows a sphere contracts homothetically to a point. In the literature on finite-time singularities, these flows have been studied mainly in the case where the speed is homogeneous of degree one, which has better analytical properties. We point out that the analogue of Theorem \ref{Hui-Sin} cannot hold in full generality when the homogeneity degree is less than one. In fact if $f=K^{\frac{1}{n+2}}$, with $K$ the Gauss curvature, the flow is invariant under affine transformations \cite{A2b} and so there exist pinched ancient solutions given by homothetically shrinking ellipsoids.
 
In \S 3.1 we prove an estimate for ancient solutions satisfying a uniform curvature pinching condition of the form $\lambda_1 \geq \epsilon H$, as in Theorem \ref{Hui-Sin}. The estimate holds for speeds with general homogeneity degree and shows that the diameter growth and the curvature decay of the solution as $t \to +\infty$ have the same rate as in the case of the shrinking sphere. In \S 3.2 we consider speeds with homogeneity one that are either convex or concave, and we use the estimate of the previous section to prove that uniform pinching characterizes the shrinking sphere. In this way, we provide an alternative argument for some results in \cite{HS,LL}. In the remainder of the paper, we consider speeds with general homogeneity. In \S 3.3 we consider the flow by powers of the Gauss curvature, a widely studied flow which enjoys important monotonicity properties in terms of suitably defined entropies. Using the special features of this flow, we are able to prove that the spheres are characterized by the same uniform pinching condition as in the homogeneity one case. Finally, in \S 3.4, we examine solutions of flows for a general class of speeds with higher homogeneity. Inspired by the work of Andrews and McCoy \cite{AMC}, we obtain a rigidity result for the sphere by assuming a stronger pinching condition of the form $|h|^2 \leq (\frac 1n + \epsilon) |H|^2$, with $\epsilon$ suitably small. \medskip

{\bf Note:} After this research was completed, a related preprint by Lynch and Nguyen \cite{LN} has appeared, where the authors obtain independently similar results to those of \S 2 of our paper, using some interesting different arguments. In particular, they prove a lemma on the compactness of submanifolds with curvature pinching which allows them to characterize the shrinking sphere in Euclidean space by rescaling techniques.

\section{High codimension Mean Curvature Flow}

We begin by introducing some notation. Throughout the paper, \linebreak \hbox{$\varphi: M^n \times (-\infty,0) \rightarrow N^{n+k}$} will be a time dependent immersion of a closed, smooth, $n$-dimensional manifold $M$ into an $(n+k)$-dimensional smooth Riemannian manifold $N$. 
For any fixed time $t$, we consider the usual geometric quantities, such as the metric $g$ induced by the immersion, the second fundamental form $h$, the Weingarten operator $W$ and the volume element $d\mu$. Quantities defined on the submanifold will be denoted by latin indices and quantities defined on the ambient space by greek indices; we will also identify $M$ with $\varphi(M)$. With the choice of a basis $\left\{e_i\right\}_{i=1}^n$ for the tangent space and one $\left\{\nu_\alpha\right\}_{\alpha=1}^k$ for the normal space, the expression of $h$ in coordinates is given by $h(e_i,e_j)=h_{ij\alpha}\nu_{\alpha}$ (we will always use the summation convention on repeated indices). The mean curvature vector $\textbf{H}$ is the trace of the second fundamental form with respect to $g$, so $\textbf{H}= \textbf{H}_{\alpha}\nu_{\alpha}= g^{ij}h_{ij\alpha}\nu_{\alpha}$; we will denote the squared norm of $h$ as $|h|^2$. It is easy to see that $|h|^2 \geq |\textbf{H}|^2/n$, with equality only at umbilical points, i.e. those points where, for a suitable choice of the basis, we have $h_{ij\alpha} = \lambda \delta_{ij}$ if $\alpha=k+1$ and $h_{ij\alpha} \equiv 0$ otherwise, for some constant $\lambda$.

The codimension $k$ will be assumed greater than one in this section, while Section 3 will consider problems with $k=1$.
When the codimension is one, there is only one normal direction, and we will choose $\nu$ as the outer normal on closed embedded hypersurfaces. The mean curvature vector is given by $\textbf{H}=-H\nu$, with $H$ a smooth real-valued function on $M$. 
There is a single Weingarten operator; its eigenvalues are the principal curvatures and will be denoted by $\lambda_1 \leq \dots \leq \lambda_n$. We say that the immersion is convex (respectively, strictly convex) if $k=1$ and $\lambda_i \geq 0$ ($\lambda_i > 0$) for all $i$; if the submanifold is embedded, this definition coincides with the usual one, so $M$ is the boundary of a convex body $\Omega$.

For notational simplicity, in this section we will write $H$ instead of $\textbf{H}$ for the mean curvature vector. This is different from our notation for hypersurfaces in Section 3, where $H$ is a scalar function; however, since the computations of the two cases are independent, there will be no possibility of confusion. In addition, throughout the paper, constants (e.g. $c$) might vary from one formula to another.

\subsection{Flow in Euclidean space}

In this subsection we prove the following:

\begin{thm}\label{teo1}

Let $M_t = \varphi(M,t)$ be a closed ancient solution of \eqref{MCFH} in $\mathbb{R}^{n+k}$, with $n,k \geq 2$. Suppose that, for all $t \in (-\infty,0)$ we have $|H|^2 > 0$ and $|h|^2 \leq C_0 |H|^2$, with $C_0$ a constant satisfying
\begin{equation} \label{pinch}
C_0 < \left\{
\begin{aligned} 
		&\frac{1}{n-1} \quad\quad\, &&n \geq 4\\[5pt]
		&\frac{4}{3n} \quad\quad\, &&n= 2,3.
\end{aligned}
\right.
\end{equation}

Suppose furthermore that the norm of the second fundamental form is uniformly bounded away from the singularity, so there exists  $h_0>0$ such that $|h|^2\leq h_0$ in $(-\infty, -1)$. Then $M_t$ is a family of shrinking spheres. 
\end{thm}

In \cite{AB}, Andrews and Baker have proved that the pinching inequality \eqref{pinch} is invariant under the flow, and that submanifolds satisfying this condition at an initial time $T_0$ evolve into a ``round point'' in finite time. We recall the evolution equations for the relevant geometric quantities derived in that paper:
\begin{align}
\frac{\partial g_{ij}}{\partial t}&=-2H\cdot h_{ij} \label{mcfhmet}\\[5pt]
\frac{\partial H}{\partial t} &= \Delta H + H \cdot h_{pq}h_{pq}\\[5pt]
\frac{\partial |h|^2}{\partial t} &= \Delta |h|^2-2|\nabla h|^2 +2R_1\\[5pt]
\frac{\partial |H|^2}{\partial t}&=\Delta |H|^2-2|\nabla H|^2+2R_2
\end{align}
with
\begin{align}
&R_1 = \sum\limits_{\alpha,\beta}\left(\sum\limits_{i,j}h_{ij\alpha}h_{ij\beta}\right)^2+\sum\limits_{i,j,\alpha,\beta}\left(\sum\limits_p h_{ip\alpha}h_{jp\beta}-h_{jp\alpha}h_{ip\beta}\right)^2 \label{erreuno}\\[5pt]
&R_2 = \sum\limits_{i,j}\left(\sum\limits_{\alpha}H_{\alpha}h_{ij\alpha}\right)^2.\label{erredue}
\end{align}

As in \cite{AB,HS}, for fixed small $\sigma>0$ we consider the function

\begin{equation}
f_\sigma = \frac{|h|^2 - \frac{1}{n} |H|^2}{|H|^{2(1-\sigma)}},
\end{equation}
and we observe that, for any $\sigma$, $f_{\sigma}$ vanishes at $x \in M$ if and only if $x$ is an umbilical point. Therefore, if $f_\sigma=0$ everywhere on $M_t$ for some $t$, then $M_t$ is a totally umbilical submanifold, hence an $n$-dimensional sphere in $\mathbb{R}^{n+k}$. As spheres evolve by homothetic shrinking, $f_{\sigma}$  will remain zero for all subsequent times. Thus, to obtain Theorem \ref{teo1}, it is enough to show that $f$ is identically zero on some time interval $(-\infty, T_1]$, with $T_1 < 0$. To this purpose, we prove the following estimate.

\begin{prop}\label{pro1}
Under the hypotheses of Theorem \ref{teo1}, there are constants $\alpha,\beta>0$ depending only on $n, C_0$ and $C=C(C_0,n,h_0) >0$ such that, for all $[T_0, T_1] \subset (-\infty, -1)$ and for all $p > \alpha$, $\sigma \leq\frac{\beta}{\sqrt{p}}$, $\sigma p > n$, we have
\begin{equation}\label{interior}
\left(\int_{M_t} f_\sigma^p\right)^{\frac{1}{\sigma p}} \leq \frac{C}{|T_0|^{1-\frac{n}{\sigma p}}-|t|^{1-\frac{n}{\sigma p}}}\quad\quad \forall t \in (T_0,T_1].
\end{equation}
\end{prop}
The proposition immediately implies Theorem \ref{teo1}. Indeed, sending $T_0$ to $-\infty$ in \eqref{interior}, we obtain that $f_{\sigma}^p$ is zero for every $t<T_1$ for suitable values of $\sigma$ and $p$. $M_t$ is then a family of shrinking spheres.

We now prove Proposition \ref{pro1}.
\begin{proof}
The first part of the proof follows the strategy of \cite{HS} together with the estimates of \cite{AB}. If we set
$$
\epsilon_{\nabla} = \frac{3}{n+2} - C_0,
$$
where $C_0$ is the constant in our pinching assumption \eqref{pinch}, then $\epsilon_\nabla$ is positive and Proposition 13 in \cite{AB} states that
\begin{align}\label{eqfsigma}
\frac{d}{dt} \int_{M_t} f_\sigma^p\, d\mu_t \,\leq\,&- \frac{p(p-1)}{2}\int_{M_t} f_\sigma^{p-2}|\nabla f_\sigma|^2\, d \mu_t  \\[7pt]
&- p \epsilon_\nabla \int_{M_t}\frac{f_\sigma^{p-1}}{|H|^{2(1-\sigma)}}|\nabla H|^2\, d\mu_t + 2 p \sigma \int_{M_t} |H|^2 f_\sigma^p \,d\mu_t \nonumber
\end{align}
for any $p \geq \max \left\{ 2, \frac{8}{\epsilon_\nabla + 1} \right\}$.

In addition, Proposition 12 of the same paper shows that there exists a constant $\epsilon_0$ depending only on $C_0$ and $n$ such that
\begin{align*}
\int_{M_t} |H|^2 f_\sigma^p\, d\mu_t \leq &\,\frac{p \eta + 4}{\epsilon_0}\int_{M_t} \frac{f_\sigma^{p-1}}{|H|^{2(1-\sigma)}}|\nabla H|^2\, d\mu_t\\[7pt]
& + \frac{p-1}{\epsilon_0 \eta}\int_{M_t}f_\sigma^{p-2}|\nabla f_\sigma|^2\, d\mu_t
\end{align*}
for all $p \geq2$, $\eta > 0$. If we fix $\eta = \frac{8 \sigma}{\epsilon_0}$ and we take any $p,\sigma$ such that $p > \frac{16}{\epsilon_\nabla}$, $\sigma \leq\frac{\epsilon_0}{8}\sqrt{\frac{\epsilon_\nabla}{p}}$, we obtain
\begin{align*}
4p\sigma \int_{M_t} |H|^2 f_\sigma^p\, d\mu_t \leq &\,\left( \frac{32 \sigma^2 p}{\epsilon_0^2}  + \frac{16 \sigma}{\epsilon_0} \right) p \int_{M_t} \frac{f_\sigma^{p-1}}{|H|^{2(1-\sigma)}}|\nabla H|^2\, d\mu_t\\[7pt]
& + \frac{p(p-1)}{2}\int_{M_t}f_\sigma^{p-2}|\nabla f_\sigma|^2\, d\mu_t \\[7pt]
\leq &\,\left( \frac{\epsilon_\nabla}{2}  + \frac{\epsilon_\nabla}{2} \right) p \int_{M_t} \frac{f_\sigma^{p-1}}{|H|^{2(1-\sigma)}}|\nabla H|^2\, d\mu_t\\[7pt]
& + \frac{p(p-1)}{2}\int_{M_t}f_\sigma^{p-2}|\nabla f_\sigma|^2\, d\mu_t,
\end{align*}
so that \eqref{eqfsigma} implies
\begin{equation}\label{psigma}
\frac{d}{dt} \int_{M_t} f_\sigma^p\, d\mu_t\, \leq \, -2p\sigma \int_{M_t} |H|^2 f_\sigma^p\, d\mu_t
\end{equation}
for all $p > \frac{16}{\epsilon_\nabla}$, $\sigma \leq\frac{\epsilon_0}{8}\sqrt{\frac{\epsilon_\nabla}{p}}$.

Thanks to the definition and our pinching assumption, we have
\begin{equation*}
0\leq f_\sigma \leq (C_0 - \frac{1}{n})|H|^{2\sigma} \leq |H|^{2\sigma}\qquad
\end{equation*}
so we obtain
\begin{align}\label{volest}
\frac{d}{dt} \int_{M_t} f_\sigma^p\, d\mu_t 
&\leq -2p\sigma \int_{M_t} f_\sigma^{p+\frac{1}{\sigma}} d \mu_t \notag \\[7pt]
&\leq -2p\sigma  \left( \int_{M_t} f_\sigma^p d\mu_t \right)^{1+\frac{1}{\sigma p}} \cdot |M_t|^{-\frac{1}{\sigma p}}
\end{align}
using H{\"o}lder's inequality, where $|M_t|$ is the volume of $M_t$. 
\begin{cla*}
There exists a constant $C=C(n,h_0)$ such that
\begin{equation}\label{claimvol}
|M_t| \leq C|t|^n, \mbox{ for all }t \leq -1.
\end{equation}
\end{cla*}

Once the claim is proved, the statement of the proposition follows easily. In fact, setting $\psi(t) = \int_{M_t} f_\sigma^p d\mu_t$, and using \eqref{volest}, we have
\begin{equation*}
\frac{d}{dt} \psi^{-\frac{1}{\sigma p}} =-\frac{1}{\sigma p} \psi^{-\left(\frac{1}{\sigma p} +1\right)} \frac{d}{dt} \psi \geq C(|t|)^{-\frac{n}{\sigma p}}.
\end{equation*}
As $\psi(t)\neq 0$ implies $\psi(s)\neq 0$ for $s<t$, we obtain, integrating on a time interval $(T_0,t]$, with $t\leq T_1$,
\begin{align*}
\psi^{-\frac{1}{\sigma p}}(t) &\geq \psi^{-\frac{1}{\sigma p}}(T_0)+C \int_{|t|}^{|T_0|} \tau^{-\frac{n}{\sigma p}} d\tau > C \int_{|t|}^{|T_0|} \tau^{-\frac{n}{\sigma p}}d\tau \\[5pt]
&> C \left(|T_0|^{1-\frac{n}{\sigma p}}-|t|^{1-\frac{n}{\sigma p}}\right)\label{integreq} 
\end{align*}
as $\sigma p > n$.

It remains to prove claim \eqref{claimvol}. For this part, we cannot adapt the technique of \cite{HS}, which only applies to hypersurfaces, and we use a different argument.

We first recall a result by Chen \cite{C}, which gives a lower bound for the minimum of sectional curvatures $K_{\pi}(p)$ under our pinching assumption:
\[
\min_{\pi \subset Gr(2,T_p M)} K_{\pi}(p) \geq \frac{1}{2}\left(\frac{1}{n-1}-C_0\right)|H|^2(p). \qquad\qquad \label{chen}
\]
In particular, this implies that all the evolving submanifolds have nonnegative Ricci curvature, so by Bishop-Gromov's theorem we can bound the volume of balls of arbitrary radius in $M_t$ with the volume of balls in $\mathbb{R}^n$. In particular, if $d_t$ is the intrinsic diameter of $M_t$, we have
\begin{equation} \label{volspher}
|M_t| \leq Vol(B_{d_t}) = \theta d_t^n,
\end{equation}
where $\theta$ is a constant only depending on $C_0$ and $n$.

To bound the intrinsic diameter, we estimate the change of the distance between two points during the evolution, using a technique inspired by the one of \cite[\S 17]{H} for the Ricci flow.

Let $P$ and $Q$ be two fixed points in $M$ and let $\gamma$ be a curve from $P$ and $Q$. The length of $\gamma$ varies with time along the flow, and we denote it by $L[\gamma_t]$. The evolution equation \eqref{mcfhmet} for the metric gives
\begin{equation*}
\frac{d}{dt} L[\gamma_t] = - \int_{\gamma} H^\alpha h_{ij\alpha} \gamma^{'i}\gamma^{'j} ds,
\end{equation*}
where $\gamma$ is parametrized by arclength at time $t$ and $\gamma'$ is the tangent vector.
By the same proof as in \cite[Lemma 17.3]{H}, we can estimate the time derivative of the distance from $P$ and $Q$ as follows. Let $\Gamma$ denote the set of all distance minimizing geodesics between $P$ and $Q$ at time $t$. Then

\begin{equation}\label{estdist}
-\sup_{\gamma \in \Gamma} \int_{\gamma_t} H^\alpha h_{ij\alpha}\gamma^{'i}\gamma^{'j} ds \leq \frac{d}{dt} d(P,Q,t) \leq -\inf_{\gamma \in \Gamma} \int_{\gamma_t} H^\alpha h_{ij\alpha}\gamma^{'i}\gamma^{'j} ds.
\end{equation}
We have
$|H^\alpha h_{ij\alpha}\gamma^{'i}\gamma^{'j}| \leq |H||h| \leq C_0|H|^2$.
As estimate \eqref{chen} implies, for the Ricci curvature,
\begin{equation*}
Ric(\gamma',\gamma') \geq (n-1)\frac{1}{2}\left(\frac{1}{n-1}-C_0\right)|H|^2,
\end{equation*}
we also have
\begin{equation*}
H^\alpha h_{ij\alpha}\gamma^{'i}\gamma^{'j} \leq C Ric(\gamma',\gamma')
\end{equation*}
along $\gamma$, with $C$ a constant depending only on $C_0$ and $n$.

We now recall Theorem 17.4 from \cite{H2}, which states the following. If the Ricci curvature of a $n$-dimensional manifold is nonnegative, for any curve $\gamma$ of length $L$, for any $v \in (0, \frac{L}{2})$,
\begin{equation*}
\int_v^{L-v} Ric(\gamma',\gamma') \, ds \leq \frac{2(n-1)}{v}.
\end{equation*}
We observe that the norm of the Ricci tensor of the submanifold is uniformly bounded away from the singular time, as a consequence of the boundedness of the second fundamental form and the Gauss equations. Thus, there is $R>0$ such that $|Ric(p,t)| \leq R$ for $t \leq -1$. Then \eqref{estdist} implies
\begin{align*}
 - \frac{d}{dt} d(P,Q,t) & \leq  \int_0^v CR \, ds + \frac{2(n-1)}{v} + \int_{L-v}^L CR \, ds \\
& \leq  2C R v + \frac{2(n-1)}{v}.
\end{align*}

Since our aim is to estimate the diameter of $M_t$ for large negative times, it is not restrictive to assume that $d(P,Q,t)>2$. Then we can choose $v=1$ to obtain 
\begin{equation*}
- \frac{d}{dt}\,d(P,Q,t)\, \leq  2C R + 2(n-1) =: C'
\end{equation*}
and integrating on $[t,-1]$ for an arbitrary $t <-1$  we get
\begin{equation*}
d(P,Q,t)\leq d(P,Q,-1) + C'(|t|-1)
\end{equation*}
which implies
\begin{equation}
d_t \leq d_{-1} + C' (|t|-1) \leq C'' |t| \quad\quad \forall t \leq -1.
\end{equation}
Using \eqref{volspher}, we conclude that $|M_t| \leq C''|t|^n$, proving our claim \eqref{claimvol} and the proposition.
\end{proof}

\subsection{High codimension Mean Curvature Flow in the sphere}

We want to characterize the same phenomenon on spheres $\mathbb{S}_K^{n+k}$ of constant sectional curvature $K$, showing that any ancient solution of the Mean Curvature Flow that satisfies a uniform pinching condition on the sphere is a family of shrinking totally umbilical submanifolds; we will refer to them with the term ``spherical caps''. We also prove that we can relax the condition to admit points with $|H|=0$ and the ancient solution will be either a spherical cap or a totally geodesic submanifold.

\begin{thm}

Let $M^n_t$ be a closed ancient solution of \eqref{MCFH} in $\mathbb{S}^{n+k}_K$. 
\begin{enumerate}
	\item
	If, for all $t \in (-\infty,0)$, 
 $0<|h|^2 \leq \frac{4}{3n} |H|^2$ holds,
then $M_t$ is a shrinking spherical cap.
\item If, for all $t \in (-\infty,0)$,  we have $|h|^2 \leq \alpha|H|^2 + \beta K$, with
\begin{equation*}
\left\{
\begin{aligned} 
		& \alpha= \frac{1}{n-1} \quad\, && \beta = 2 &&n \geq 4\\[5pt]
		& \alpha= \frac{4}{9} \quad\, && \beta=\frac{3}{2} \quad \quad&&n=3 \\[5pt]
		& \alpha= \frac{2}{4-\beta} && \beta <\frac{12}{13} &&n=2,
		%&\frac{4}{6+2\epsilon}\quad\, OR \quad \frac{2}{4-\beta} &&n=2
\end{aligned}
\right.
\end{equation*}

 then $M_t$ is either a shrinking spherical cap or a totally geodesic submanifold.
\end{enumerate}
\end{thm}

\begin{proof}
	
The evolution equations for metric and curvature in a spherical ambient manifold, see \cite{B}, are given by:
\begin{align}
\frac{\partial|H|^2}{\partial t} &= \Delta |H|^2 -2|\nabla H|^2 +2nK|H|^2 + 2R_2\\
\frac{\partial|h|^2}{\partial t} &= \Delta |h|^2 -2 |\nabla h|^2 + 2 R_1 + 4K|H|^2 -2nK|h|^2, \label{evohsp}
\end{align}
where $R_1$ and $R_2$ are as in \eqref{erreuno} and \eqref{erredue}.
To prove the first statement, let us define
\begin{equation*}
f_0 = \frac{|h|^2-\frac{1}{n}|H|^2}{|H|^2}.
\end{equation*}
\\
Using the evolution equations above, we find
\begin{align}
\frac{\partial f_0}{\partial t} =\,& \Delta f_0 + \frac{4H}{|H|^2}\left\langle\nabla |H|, \nabla f_0\right\rangle - \frac{2}{|H|^2}\left[|\nabla h|^2 - \left(\frac{1}{n} + f_0\right) |\nabla H|^2\right] \notag \\[5pt]
& +\frac{2}{|H|^2}\left[R_1 - \left(\frac{1}{n}+f_0\right) R_2\right] - 4nK f_0.
\end{align}
By our pinching assumption, we have $\frac{1}{n} + f_0 \leq \frac{4}{3n}<\frac{3}{n+2}$. Then the gradient terms give a nonpositive contribution, as it is well known, see e.g. \cite{B}, that
\begin{equation}\label{gradhH}
|\nabla h|^2 \geq \frac{3}{n+2}|\nabla H|^2.
\end{equation}

In order to analyse the reaction terms, we recall some more notation from \cite{B}. By assumption, $|H|^2 > 0$ everywhere, so we can choose an adapted orthonormal basis $\left\{\left\{e_i\right\}_{i=1}^n, \left\{\nu_{\alpha}\right\}_{\alpha=1}^k\right\}$ for the sphere, such that $\left\{\nu_{\alpha}\right\}_{\alpha=1}^k$ is a frame for the normal space of the submanifold with $\nu_1 = \frac{H}{|H|}$, while, if we write $\mathring{h} = h - \frac{1}{n} H \otimes g = \sum\limits_{\alpha=1}^k \mathring{h}_{\alpha}$, the frame $\left\{e_i\right\}_{i=1}^n$ is tangent and diagonalizes $\mathring{h}_1$. In addition, we denote the norm of $\mathring{h}$ in the other directions with $|\mathring{h}_-|^2$, so that $|\mathring{h}|^2 = |\mathring{h}_1|^2 + |\mathring{h}_-|^2$. 

With this choice we have
\begin{equation}
\label{eqr2}
R_2 = |\mathring{h}|^2 |H|^2 + \frac{1}{n} |H|^4,
\end{equation}
and the following estimate holds, proved in \cite[\S 5.2]{B},
\begin{equation}
\label{eqr1}
R_1-\frac 1n R_2 \leq |\mathring{h}_1|^4 + \frac 1n |\mathring{h}_1|^2 |H|^2 +
4 |\mathring{h}_1|^2|\mathring{h}_-|^2 + \frac{3}{2} |\mathring{h}_-|^4.
\end{equation}
From this, we obtain
\begin{align*}
2 \left[R_1 - \left(\frac{1}{n}+ f_0 \right)R_2 \right] &\leq 2|\mathring{h}_1|^4 + \left(\frac{2}{n} -2 f_0 \right)|\mathring{h}_1|^2|H|^2 - \frac{2}{n} f_0 |H|^4 \\[5pt]
 &+8 |\mathring{h}_1|^2|\mathring{h}_-|^2 + 3 |\mathring{h}_-|^4.
\end{align*}

Using the definition of $f_0$ and substituting $|\mathring{h}|^2 = |\mathring{h}_1|^2 + |\mathring{h}_-|^2$, some terms simplify and the right hand side can be rewritten as
$$
 |\mathring{h}_-|^2\left(6|\mathring{h}_1|^2-\frac{2}{n}|H|^2 + 3|\mathring{h}_-|^2\right).
$$
Thus, using our pinching assumption, we conclude
\begin{align}
2R_1 - 2\left(\frac{1}{n}+f_0 \right)R_2 & \leq
 |\mathring{h}_-|^2\left(6|\mathring{h}_1|^2-\frac{2}{n}|H|^2 + 3|\mathring{h}_-|^2\right)\notag\\[5pt]
& \leq 2|\mathring{h}_-|^2\left(3|h|^2 - \frac{4}{n} |H|^2\right) \leq 0.
\end{align}

So we can proceed as in the first case of \cite{HS}; we have
\begin{equation}\label{expsfera}
\frac{\partial f_0}{\partial t} \leq \Delta f_0 + \frac{4H}{|H|^2}\left\langle\nabla f_0, \nabla H\right\rangle -4nK f_0.
\end{equation}
If there existed $t_1 < 0$ such that $f_0(t_1)$ is not identically zero on $M_{t_1}$, we could apply the maximum principle to get
\begin{equation}
0 < \max_{M_{t_1}} f_0 \leq e^{-4nK(t_1-t)}\max_{M_t} f_0 \;\;\; \forall t<t_1
\end{equation}
and the function would explode for $t \rightarrow -\infty$, contradicting our assumption  \hbox{$f_0 \leq \frac{4}{3n} - \frac{1}{n}$}.
So $f_0$ is identically zero for all times and $M_t$ is a family of totally umbilical submanifolds of the sphere.

For the second case, we do not require $|H|^2 \neq 0$; we want to show that $M_t$ is a shrinking spherical cap or an equator.

We can follow a computation similar to par. 5.3 in \cite{B} and consider a perturbed pinching function \hbox{$f=\frac{|\mathring{h}|^2}{a|H|^2+bK}$} where $a=\alpha - \frac{1}{n}$ and the constant $b$ is given by:
$$
b=\left\{
\begin{aligned}
	&\;\frac{11}{10}\quad\, &&n \geq 4\\[5pt]
	&\;\frac{33}{40} \quad\, &&n=3\\[5pt]
	&\frac{24\beta}{13(4-\beta)} &&n=2.
\end{aligned}
\right.
$$

The function $f$ satisfies the equation
\begin{align}
\frac{\partial}{\partial t} f &= \Delta f + \frac{2a}{a|H|^2+bK}\left\langle \nabla_i |H|^2, \nabla_i f\right\rangle  \nonumber\\
&-\frac{2}{a|H|^2+bK}\left(|\nabla h|^2-\frac{1}{n}|\nabla H|^2 - \frac{a|\mathring{h}|^2}{a|H|^2+bK}|\nabla H|^2\right) \label{eqnf}\\
&+ \frac{2}{a|H|^2+bK}\left(R_1 -\frac{1}{n} R_2 - n K|\mathring{h}|^2-\frac{aR_2|\mathring{h}|^2}{a|H|^2+bK}-\frac{anK|\mathring{h}|^2|H|^2}{a|H|^2+bK}\right). \nonumber
\end{align}
Using \eqref{gradhH} and $b \leq \beta$, we can estimate the gradient terms as follows:
\begin{align*}
-&\left(|\nabla h|^2-\frac{1}{n}|\nabla H|^2 - \frac{a|\mathring{h}|^2}{a|H|^2+bK}|\nabla H|^2\right)\\ &\;\leq -\left(\frac{2(n-1)}{n(n+2)}-\frac{a^2 |H|^2 + a\beta K}{a|H|^2 +bK} \right)|\nabla H|^2\\
&\;\leq - \left(\frac{2(n-1)}{n(n+2)}-\frac{a\beta}{b}\right)|\nabla H|^2,
\end{align*}
and this expression is negative for the chosen values of $b$. 

At the points with $|H|^2 \neq 0$, we can employ the adapted frame to analyse the reaction terms in \eqref{eqnf}. We set
$$
\tilde R = \left(a|H|^2+bK\right)\left(R_1-\frac{1}{n}R_2 - nK|\mathring{h}|^2\right) - aR_2|\mathring{h}|^2-anK|\mathring{h}|^2|H|^2
$$
and we use \eqref{eqr2}--\eqref{eqr1} to obtain
\begin{align*}
\tilde R  \leq & \ a |H|^2 \left(3|\mathring{h}_1|^2|\mathring{h}_-|^2+\frac{3}{2}|\mathring{h}_-|^4 -\frac{1}{n}|\mathring{h}_-|^2|H|^2 \right) \\
&+ bK\left(|\mathring{h}_1|^4+\frac{1}{n}|\mathring{h}_1|^2|H|^2+4|\mathring{h}_1|^2|\mathring{h}_-|^2+\frac{3}{2}|\mathring{h}_-|^4\right) \\
& - nK|\mathring{h}|^2(2a|H|^2+bK).
\end{align*}
We use a Peter-Paul inequality to estimate
$$
|\mathring{h}_1|^4+4|\mathring{h}_1|^2|\mathring{h}_-|^2+\frac{3}{2}|\mathring{h}_-|^4
\leq \frac 53 |\mathring{h}_1|^4 + \frac {10}3 |\mathring{h}_1|^2|\mathring{h}_-|^2 + \frac 53 |\mathring{h}_-|^4 = \frac{5}{3}|\mathring{h}|^4.
$$
Thanks to our assumptions, we have $|H|^2 \geq \frac{|\mathring{h}|^2-\beta K}{a}$, $a \leq \frac{1}{3n}$ and $b \leq \beta$. Using this, we find
\begin{align*}
\tilde R  \leq & \ a|H|^2|\mathring{h}_-|^2\left(3|\mathring{h}|^2-\frac{1}{na}\left(|\mathring{h}|^2-\beta K\right)\right) \\
& + \frac 53 bK  |\mathring{h}|^4 +\frac{bK}{n}|\mathring{h}_1|^2|H|^2 - nK|\mathring{h}|^2(2a|H|^2+bK) \\
 \leq & \ |H|^2 \left(\frac{\beta K}{n}|\mathring{h}|^2-2anK|\mathring{h}|^2\right) +\frac{5}{3}bK|\mathring{h}|^4 -nbK^2|\mathring{h}|^2.
\end{align*}

Since our choice of the constants implies that $\beta/n < 2an$, we can use again $|H|^2 \geq \frac{|\mathring{h}|^2-\beta K}{a}$ and $\beta \geq b$ to find that,
for any small $\epsilon>0$
\begin{align*}
\tilde R \leq &
 -\epsilon a K |H|^2 |\mathring{h}|^2
-\frac{1}{a}(|\mathring{h}|^2-\beta K)\left(2anK - \epsilon aK-\frac{\beta K}{n}\right) |\mathring{h}|^2
  \\
  & +\frac{5}{3}bK|\mathring{h}|^4-nbK^2|\mathring{h}|^2 \\
\leq & 
-\left[2n-\epsilon-\frac{\beta }{na}-\frac{5}{3}b\right] K |\mathring{h}|^4-\left[\frac{\beta^2}{na}+nb -2n\beta\right] K^2|\mathring{h}|^2 \\
& -\epsilon K(a |H|^2 +bK) |\mathring{h}|^2.
\end{align*}
With our choice of $a, \beta, b$, we can check that both expressions in square parentheses are negative if $\epsilon$ is suitably small. The above estimates, together with \eqref{eqnf}, imply the inequality
\begin{equation*}
\frac{\partial f}{\partial t} \leq \Delta f + \frac{2a}{a|H|^2+bK}\left\langle \nabla_i |H|^2, \nabla_i f\right\rangle -2 \epsilon K f.
\end{equation*}

It remains to consider the points where $H=0$. In this case we have $f= \frac{|h|^2}{b}$ and the reaction terms in \eqref{eqnf} become
$$
\frac{2}{b}\{R_1 - nK|h|^2\}.
$$
As in \cite[Lemma 5.1]{B}, $R_1 \leq \frac{3}{2}|h|^4$ holds whenever $H=0$, so that
$$
2R_1 - 2nK|h|^2  \leq \left(3 \beta - 2n \right) K |h|^2=  \left(3 \beta - 2n \right) Kbf.
$$
Since $\beta<\frac{2n}{3}$, we conclude that the reaction terms are bounded above by a negative multiple of $f$ also at these points. Then, as in the first part of the theorem, we can apply the maximum principle and find a contradiction unless $f \equiv 0$. Since we are allowing $H=0$, the solution can be either a shrinking spherical cap or a totally geodesic submanifold.
\end{proof}

\section{Hypersurface flows with a general speed}

In this section, we turn our attention to ancient solutions of flows of hypersurface immersions of $M$ into $\mathbb{R}^{n+1}$ of the form 
\begin{equation}
\frac{\partial \varphi}{\partial t}(x,t) = -F(W(x,t))\nu(x,t) = -f(\lambda(W(x,t)))\nu(x,t)
\end{equation}
where $\lambda$ is the function that associates to a self-adjoint operator its ordered eigenvalues. We assume throughout
\begin{description}
\item[(H1)] $f:\Gamma \rightarrow \mathbb{R}$ is a symmetric smooth function, homogeneous of degree $\alpha$ for some $\alpha>0$, defined on an open symmetric cone $\Gamma \subset \mathbb{R}^n$ containing the positive cone $\Gamma_+$.
\item[(H2)]  $f$ satisfies
$$
\frac{\partial{f}}{\partial{\lambda_i}} > 0 \mbox{ in $\Gamma$ for all $i =1,\dots,n$.}
$$
 \end{description}
We remark that homogeneity and monotonicity imply that $f$ is also strictly positive on $\Gamma_+$, thanks to the Euler relation.
It is well known (see, for example, \cite{Ger}) that, under the above assumptions, $F$ is also smooth and homogeneous of degree $\alpha$ as a function of the components of $W$. In addition, if $f$ is a convex (concave) function, then the same holds for $F$.
We will denote the derivatives of $F$ by a dot, so that
\begin{align*}
\frac{d}{ds}F(A+sB)|_{s=0}=\dot{F}^{ij}|_AB_{ij}\\[5pt]
\frac{d^2}{ds^2}F(A+sB)|_{s=0}=\ddot{F}^{ij,kl}|_AB_{ij}B_{kl}
\end{align*}
for any $A,B$ symmetric matrices such that $\lambda(A)$ and $\lambda(B)$ belong to $\Gamma$.
Assumption (H2) ensures that $\dot{F}$ is positive definite and that \eqref{flowh} is a parabolic system. We denote by $\mathcal{L}$ the elliptic operator on $C^{\infty}(M)$ defined as $\mathcal{L}=\dot{F}^{ij} \nabla_i\nabla_j$. 

We will consider a convex ancient solution of \eqref{flowh} satisfying a uniform pinching condition on the principal curvatures of the following form: there exists $C_1>0$ such that
\begin{equation}
\lambda_n<C_1\lambda_1, \qquad
\mbox{for all $t \in (-\infty, 0)$, for all $x \in M_t$}.
\label{condpinch}
\end{equation}

We first present an estimate on the speed and the diameter of a pinched ancient solution which holds under very general assumption on $F$. To this purpose, we recall some definitions. If $M$ is a $n$-dimensional embedded submanifold of $\mathbb{R}^{n+1}$ bounding a convex body $\Omega$,
the inner and outer radii of $M$ are defined respectively as
\begin{align*}
\rho_-&=\sup\left\{r\, |\, B_r(y) \subset \Omega \,\, \text{for some } y \in \mathbb{R}^{n+1}\right\}\\
\rho_+&=\inf\left\{r\,|\, \Omega \subset B_r(y)\,\, \text{for some } y \in \mathbb{R}^{n+1}\right\}.
\end{align*}
Along our flow, these quantities depend on time and will be denoted by $\rho_\pm(t)$.

By a result in \cite{A1}, if the pinching condition \eqref{condpinch} holds, then there also exists $\bar C_1=\bar C_1(C_1,n)$ such that
\begin{equation}\label{ratio}
\rho_+(t) \leq \bar C_1 \rho_-(t), \qquad \forall t<0.
\end{equation}

\subsection{A general estimate}

We can prove that the pinching condition \eqref{condpinch} implies strong bounds on the inner and outer radii and on the speed of an ancient solution. The result holds for general homogeneous speeds and can be proved by a well known technique first introduced in \cite{T}.

\begin{thm}\label{teorbound0}
Let $M_t=\varphi(M,t)$ be a convex ancient solution of the flow \eqref{flowh} defined for $t \in (-\infty,0)$ and shrinking to a point as $t \to 0$. Suppose that the pinching condition \eqref{condpinch} is satisfied for some $C_1>0$.
Then there exist constants $C_2,C_3$ such that for all $t \in (-\infty,-1)$:

\begin{equation}\label{diameter}
C_2^{-1} \rho_+(t) \leq |t|^{\frac{1}{\alpha+1}} \leq  C_2 \rho_-(t)\,,
\end{equation}

\begin{equation}\label{supF} 
\sup F(\cdot, t)  \leq C_3 |t|^{-\frac{\alpha}{1+\alpha}}\,.
\end{equation}
\end{thm}
\begin{proof}
For simplicity, we assume that $f$ is normalized in order to satisfy $f(1,1,\dots,1)=1$. Then the spherical solution of \eqref{flowh} which shrinks to a point at time $t=0$ has a radius given by $(\alpha+1)|t|^{\frac{1}{\alpha +1}}$. By comparison, we deduce
\begin{equation}\label{compspheres}
\rho_-(t) \leq (\alpha+1)|t|^{\frac{1}{\alpha +1}} \leq \rho_+(t), \qquad \forall t<0.
\end{equation}
Combining these inequalities with \eqref{ratio}, we immediately obtain \eqref{diameter}. 

To bound $\sup F$ from above, we use a well-known technique first introduced in \cite{T}. We fix any $t_0 < 0$ and we call $z_0$ the center of a ball realizing $\rho_-(t_0)$. We denote by $u(x,t)=\left\langle \phi(x,t)- z_0, \nu(x,t)\right\rangle $  the support function centered at $z_0$. By convexity, we have $u(\cdot,t_0) \geq \rho_-(t_0)$ and the inequality holds for all times $t \leq t_0$, thanks to the shrinking nature of the flow.
Hence, the function
\begin{equation*}
q(x,t)=\frac{F}{2 u - \rho_-(t_0)}
\end{equation*}
is well defined for $t \in (-\infty,t_0)$.
We have the evolution equation, see e.g. \cite{AMCZ}:

$$
\left(\frac{\partial}{\partial t} - \mathcal{L}\right)q = \frac{4}{u - \rho_-(t_0)}\dot{F}^{ij}
\nabla_i u \nabla_j q
+\frac{F((1+\alpha)F-\rho_-(t_0)\dot{F}^{ij}h_{ik} h^k_j)}{u - (\rho_-(t_0))^2}.
$$
The pinching condition allows to estimate
$$
\dot{F}^{ij} h_{ik} h^k_j = 
\sum_i\frac{\partial{f}}{\partial \lambda_i}\lambda_i^2 
\geq \lambda_1 \sum_i\frac{\partial{f}}{\partial \lambda_i}\lambda_i
= \alpha \lambda_1 F.
$$
Let us set $|\lambda|=\sqrt{\sum_i \lambda_i^2}$. Define
$$
K=\max \left\{ f(\lambda_1,\dots,\lambda_n)  ~:~ 0< \lambda_n \leq C_1 \lambda_1, \ |\lambda|=1 \right\}.
$$
By homogeneity, we have
\begin{equation}\label{conspinch}
f(\lambda_1,\dots,\lambda_n) \leq K |\lambda|^\alpha \leq K (\sqrt n  C_1 \lambda_1)^\alpha,
\end{equation}
for all $(\lambda_1,\dots,\lambda_n) \in \Gamma_+$ such that $\lambda_n \leq C_1 \lambda_1$. It follows
\begin{equation}
\label{quadr}
\dot{F}^{ij} h_{ik} h^k_j \geq C F^{1 + \frac 1\alpha},
\end{equation}
where we denote by $C$ any constant only depending on $n,C_1$.
 If we define $Q(t) = \sup_{M_t} q(x,t)$, we obtain the inequality
\begin{equation}
\frac{d}{dt}Q \leq Q^2(1+\alpha-C\rho_-(t_0)F^{\frac 1\alpha})\leq Q^2(1+\alpha-C\rho_-(t_0)^{1+\frac 1\alpha}Q^{\frac 1\alpha}).
\end{equation}
It follows that $Q(t) \geq \psi(t)$ for all $t \leq t_0$, where $\psi$ is the solution to the equation
\begin{equation}
\frac{d}{d t} \psi = \psi^2(1+\alpha-C\rho_-(t_0)^{1+\frac 1\alpha}\psi^{\frac 1\alpha}), \qquad t \leq t_0,
\end{equation}
with final datum $\psi (t_0) = Q(t_0)$. It is easily seen that, if $\psi(t_0)$ is such that the left hand side is negative, that is, if
\begin{equation}\label{contrad}
Q(t_0)^{\frac 1\alpha} > \frac{1+\alpha}{C\rho_-(t_0)^{1+\frac 1\alpha}},
\end{equation}
then $\psi(t)$ is decreasing for all $t<t_0$ and blows up at a finite time, since the right hand side is superlinear in $\psi$.
On the other hand, $Q(t)$ is defined for all $t \in (-\infty,t_0]$ and we obtain a contradiction. It follows that \eqref{contrad} cannot hold and that the reverse inequality is satisfied. This implies, by the definition of $F$ and by estimate \eqref{diameter},
$$
\max F(\cdot,t_0) \leq 2 \max u(\cdot,t_0) Q(t_0) \leq C \frac{\rho_+(t_0)}{\rho_-(t_0)^{\alpha+1}} \leq C |t|^{- \frac {\alpha}{\alpha+1}}.  
$$
Since $t_0 <0$ is arbitrary, this completes the proof of the inequality in \eqref{supF}.
\end{proof}

In the next section, we will show how the above result can be used to prove that pinched ancient solutions are spheres when the degree of homogeneity is one.

\begin{rmk}
In the previous theorem, the pinching hypothesis \eqref{condpinch} can be replaced by assuming a priori that the solution satisfies a bound of the form \eqref{ratio}, and that the $\alpha$-root of the speed function $f$ is inverse concave on the positive cone, that is, the function
$$
(\rho_1,\dots,\rho_n) \to f^{-\frac 1\alpha} \left( \rho_1^{-1}, \dots, \rho_n^{-1} \right)
$$
is concave. In fact, in this case property \eqref{quadr} holds even without the pinching assumption, see Lemma 5 in \cite{AMCZ}, and the same proof applies.
\end{rmk}

\subsection{Degree of homogeneity 1}

Let us now restrict to the case where the speed is homogeneous of degree one. We show here that the estimates of the previous section allows us to give a quick proof of the result that ancient pinched solutions are shrinking spheres, thus providing an alternative approach to the results of \cite{HS} and \cite{LL}.

We will require either convexity or concavity of the speed, since this will allow us to apply Krylov-Safonov's regularity theory for fully nonlinear parabolic equations  \cite{K} to the equations associated with our flow. An easy case of a convex speed is $f=|\lambda|$, while classical concave examples are $f=S_k^{1/k}$ or $f=S_k/S_{k-1}$, where $S_k$ are the elementary symmetric polynomials of the principal curvatures.

\begin{thm}\label{teorho1}
Let $M_t=\varphi(M,t)$ be an ancient solution of the flow \eqref{flowh}, with $f$ satisfying {\rm (H1)--(H2)} with $\alpha=1$. Suppose in addition that $f$ is either convex or concave on the positive cone, and that $M_t$ satisfies \eqref{condpinch}. Then $M_t$ is a family of shrinking spheres.
\end{thm}

\begin{proof}
We first observe that, under our hypotheses, our solution shrinks to a point at the singular time by the results of \cite{A1}. In that paper, additional assumptions are made in order to obtain the preservation of pinching, but we do not need them here because we are assuming pinching a priori; we observe that the estimates of Theorem \ref{teorbound0} apply. We adapt to our setting the procedure of \S 7 in \cite{A1}, with the difference that we want to prove convergence to a spherical profile backwards in time rather than forward; we will omit the details that are entirely analogous to \cite{A1}.

As usual, we assume that the singular time is $t=0$. We  consider a rescaling of the solution, choosing a new time variable \hbox{$\tau = -\frac{1}{2}\log(-t)$}, so that $\tau \in (-\infty, + \infty)$ and define immersions \hbox{$\widetilde{\varphi}_{\tau} = (-2t)^{-\frac{1}{2}}(\varphi_t - p)$}, where $p$ is the limit point of original immersions; quantities pertaining to rescaled solutions will be denoted with a tilde. By Theorem \ref{teorbound0}, $\tilde \rho_\pm$ and $\sup \widetilde{F}$ are bounded from above and below uniformly for all $\tau$. Then, as in Lemma 7.7 in \cite{A1}, we can write the rescaled solution as a spherical graph and apply Krylov-Safonov's Harnack inequality to show that $\min \widetilde{F}$ is bounded away from zero, which ensures that the curvatures of the rescaled solution stay in a compact subset of the positive cone and that the flow is uniformly parabolic.

Then, as in Lemma 7.9 in \cite{A1}, we can apply Krylov-Safonov's regularity results to show that the support function $\widetilde{u}$ of the rescaled immersion is uniformly bounded in $C^k$ for all $k$. We can thus find sequences of times going to $-\infty$ along which $\widetilde{u}$ converges to some limit $\widetilde{u}_{-\infty}$ that is the support function of a convex, compact hypersurface. To study the structure of the possible limits, we consider a suitable zero-homogeneous function of the curvatures whose integral is monotone along the rescaled flow. We recall for instance the procedure for a concave $F$. In this case we define $\eta = \frac{|\widetilde{W}|}{\widetilde{F}}$, and we denote by $\eta_0=\eta(1,\dots,1)$. It is easy to check that $\eta \geq \eta_0$, with equality only at umbilical points. Then \cite[Lemma 7.9]{A1} gives the following estimate, for suitably large $p>0$:
\begin{equation}\label{andrews}
\frac{d}{d \tau} \int_{M} \widetilde{K}(\eta^p-\eta_0^p) d\widetilde{\mu} \leq - C_5 \int_{M}\widetilde{K}\frac{|\widetilde{\nabla}\widetilde{W}|^2}{|\widetilde{W}|^2}d\widetilde{\mu}
\end{equation}
where $\widetilde{K}$ is the rescaled Gauss curvature. We can then integrate the inequality above on an interval $[\tau_0, \tau_1]$ and send $\tau_0$ to $-\infty$: as both $\eta^p-\eta_0^p$ and $\widetilde{K}$ are bounded, positive functions, the left-hand side remains bounded, and we conclude that
$$
\int_{-\infty}^0 \left(
\int_{M}\widetilde{K}\frac{|\widetilde{\nabla}\widetilde{W}|^2}{|\widetilde{W}|^2}d\widetilde{\mu} \right) d \tau < +\infty.
$$
So we can choose a sequence of times $\left\{ \tau_k\right\}$ such that $|\widetilde{\nabla}\widetilde{W}(\cdot,\tau_k)|_{L^2(M)} \to 0$; by possibly taking a further subsequence, the corresponding support functions $\tilde u(\cdot, \tau_k)$ converge smoothly to the support function of a standard sphere. Therefore, along this subsequence, we have
$$\int_{M} \widetilde{K}(\eta^p-\eta_0^p) d\widetilde{\mu} \to 0.$$
Since by \eqref{andrews} the left hand side is positive and nonincreasing, it must be identically zero for all $\tau$. Then the solution is a family of spheres, as claimed.
\end{proof}

\begin{rmk}
Under the additional assumption of inverse-concavity of $f$, these flows admit a differential Harnack inequality \cite{A2}. Then, as in \cite{HS} and Langford-Lynch \cite{LL}, it is possible to prove that the uniform pinching assumption \eqref{condpinch} can be replaced by other hypotheses, such as
\begin{enumerate}
\item a bound on the diameter of the form diam$(M_t) \leq C(1+\sqrt{|t|})$.
\item a pinching condition on the radii $\rho_+(t) \leq C \rho_-(t)$.
\item a bound on the isoperimetric ratio $|M_t|^{n+1} \leq C |\Omega_t|^n$.
\end{enumerate}
\end{rmk}

\subsection{Flows by powers of the Gauss curvature}

We will now deal with a precise class of flows, namely
\begin{equation}\label{GCF}
\frac{\partial \varphi}{\partial t}(x,t)= -K^{\beta}(x,t)\nu(x,t)
\end{equation}
where $K$ is the Gauss curvature of the submanifold and $\beta > 0$.

Many authors in the last decades have investigated the singular behaviour of these flows. The first one was Firey \cite{Firey}, in the case $\beta=1$, who proved that a compact convex hypersurface with spherical symmetry shrinks to a round point in finite time, and conjectured that the same property holds without symmetry assumption.  After various partial results through the decades, the conjecture was proved for a general $\beta>1/(n+2)$ by combining the results of the papers \cite{AGN,DaskaChoi,BrendlDaskaChoi}, where the reader can also find more detailed references.

In this section we consider an ancient compact convex solution of \eqref{GCF} defined in $(-\infty,0)$. We translate the coordinates if necessary so that the solution shrinks to the origin as $t \to 0$. Following \cite{AGN,GN}, we consider the rescaled flow $\widetilde{\varphi}(\cdot,\tau)=e^{\tau}\varphi(\cdot,t(\tau))$, where $t$ and $\tau$ are related by
$$
\tau(t)=\frac{1}{n+1}\log\left(\frac{|B(1)|}{|\Omega_t|}\right) \,.
$$
Here $|B(1)|$ is the volume of the unit ball, which we can also write as \hbox{$|B(1)|=(n+1)^{-1}\omega_n$}, with $\omega_n=|\mathbb{S}^n|$.
In this way, the volume of the rescaled enclosed region $\widetilde \Omega_\tau$ is constant and equal to $|B(1)|$. In addition, the flow is defined for $\tau \in (-\infty,\infty)$ and satisfies the equation
\begin{equation}\label{GCFres}
\frac{\partial \widetilde{\varphi}}{\partial \tau}(x,\tau)= -\frac{\tilde K^{\beta}(x,t)}{\omega_n^{-1}\int_{\mathbb{S}^n} \tilde K^{\beta-1} d\theta}\nu(x,\tau)+\widetilde{\varphi}(x,\tau).
\end{equation}

An important feature of these flows is the existence of monotone integral quantities, called entropies, see e.g. \cite{ChowHar,Firey}.  Here we will use the ones considered in \cite{AGN}, which are defined as follows. Let $M$ be any convex embedded hypersurface in $\mathbb{R}^{n+1}$ and let $\Omega$ be the convex body enclosed by $M$.
The entropy functional $\mathcal{E}_{\beta}(\Omega)$ is defined for each $\beta>0$ as 
\begin{equation*}
\mathcal{E}_{\beta}(\Omega)=\sup_{z \in \Omega} \mathcal{E}_{\beta}(\Omega, z)\, ,
\end{equation*}
where 
$$
\mathcal{E}_{\beta}(\Omega,z)=
\left\{
\begin{aligned}
&\frac{1}{\omega_n} \int_{\mathbb{S}^n} \log u_{z}d\theta \quad\quad && \beta = 1\\
&\frac{\beta}{\beta-1}\log\left(\frac{1}{\omega_n}\int_{\mathbb{S}^n} u_z^{1-\frac{1}{\beta}} d\theta\right)\quad\quad && \beta \neq 1.
\end{aligned}
\right.
$$
Here, $u_{z}$ is the support function of $\Omega$ with respect to the center $z$, that is,  $u_{z}(\theta)=\langle \nu^{-1} (\theta) - z, \theta \rangle $ for $\theta \in \mathbb{S}^n$, where $\nu^{-1}$ is the inverse of the Gauss map. For simplicity of notation, we will use in the following the same symbol for $u_z$ considered as a function on $ \mathbb{S}^n$ and as a function defined on $M$. In \cite{AGN}, it is proved that for each $\Omega$, there exists a unique point $e \in \Omega$, called entropy point, such that the supremum in the definition is attained.

The main property of the entropy is monotonicity along the solutions of the rescaled flow \eqref{GCFres}. In fact, we have the inequality (see \cite{AGN}, Theorem 3.1):
\begin{equation}
\frac{d}{d\tau} \mathcal{E}_{\beta}(\widetilde{\Omega}_\tau) \leq
-\left[\frac{\int_{\mathbb{S}^n} f^{1+\frac 1\beta} d\sigma_\tau \cdot \int_{\mathbb{S}^n} d\sigma_\tau}
{\int_{\mathbb{S}^n} f^{\frac 1\beta} d\sigma_\tau \cdot \int_{\mathbb{S}^n} f d\sigma_\tau} -1 \right], \label{deriventro}
\end{equation}
where $f=\frac{\widetilde K^\beta}{{u}_{e(\tau)}}$, $d\sigma_\tau=\frac{{u}_{e(\tau)}}{ \widetilde K} d\theta$, and
${u}_{e(\tau)}$ is the support function of the rescaled solution at the entropy point $e(\tau)$ of $\widetilde \Omega_\tau$.

By the H\"older inequality, the right hand side of \eqref{deriventro} is nonpositive, and it is strictly negative unless $f$ is constant. The manifolds with constant $f$ are the stationary solutions of \eqref{GCFres}, which correspond to the homothetically shrinking solitons of \eqref{GCF}. Using this property, it was showed in \cite{AGN} that, for every $\beta > \frac{1}{n+2}$, convex hypersurfaces evolve into a singularity which is a soliton under rescaling. It was finally proved in \cite{DaskaChoi,BrendlDaskaChoi} that the only soliton is the sphere, thus proving Firey's conjecture.

Our result for flows by powers of the Gauss curvature is the following:
\begin{thm}
	Let $M_t$ be an ancient closed strictly convex solution of \eqref{GCF} with $\beta > \frac{1}{n+2}$. If there exists $\tilde C>0$ such that $\frac{\lambda_n}{\lambda_1}\leq \tilde C$ on $(-\infty, 0)$, then $M_t$ is a family of shrinking spheres.
\end{thm}

\begin{proof} 
	
	By our pinching assumption, the solution satisfies the conclusions of Theorem \ref{teorbound0}. Let us consider the rescaled flow \eqref{GCFres}. By construction, the domains $\widetilde{\Omega}_\tau$ contain the origin for all times. It is easy to check that estimate \eqref{supF} translates into a uniform upper bound on the Gauss curvature $\widetilde{K}$ of the rescaled hypersurfaces. By pinching, each principal curvature is also bounded. In addition, by \eqref{diameter}, the inner and outer radius are bounded from both sides by positive constants uniformly in time. Thanks to these bounds, we know from Lemma 4.4 of  \cite{AGN} that the entropy point of $\widetilde{\Omega}_\tau$ satisfies ${\rm dist}(e(\tau), \partial \widetilde{\Omega}_\tau) \geq \varepsilon_0$, for some $\varepsilon_0$ independent of $\tau$. Therefore, we have the estimates
	\begin{equation}\label{estgcf}
	\frac 1C \leq u_{e(\tau)} \leq C, \qquad \widetilde K \leq C
	\end{equation}
	on $\widetilde M_\tau$, where we denote by $C$ any large positive constant independent of $\tau$. It follows that the entropy $\mathcal{E}_{\beta}(\widetilde \Omega_\tau)$ is also bounded from above for all $\tau$. Since $\mathcal{E}_{\beta}(\widetilde \Omega_\tau)$ is monotone decreasing, it converges to some finite value $\mathcal{E}_{-\infty}$ as $\tau \to -\infty$. To conclude the proof, we need to show that $\widetilde M_\tau$ converges to a stationary point of the entropy as $\tau \to -\infty$.
	
	Using the property that
	$$
	\int_{\mathbb{S}^n} d\sigma_\tau = \int_{\widetilde M_\tau} u_{e(\tau)} d\mu = (n+1) {\rm Vol}( \widetilde \Omega_\tau) = \omega_n,
	$$
	we can rewrite formula \eqref{deriventro} as
	$$
	\frac{d}{d\tau} \mathcal{E}_{\beta}(\widetilde{\Omega}_\tau) \leq
	-\left[\frac{\int_{\widetilde M_\tau} f^{1+\frac 1\beta} d\nu }
	{\int_{\widetilde M_\tau} f^{\frac 1\beta} d\nu \cdot \int_{\widetilde M_\tau} f d\nu} -1 \right], 
	$$
	where $d\nu := \omega_n^{-1} u_{e(\tau)} d\mu$ is a probability measure on $\widetilde M_\tau$. Then, as in Proposition 4.3 in \cite{Sin}, we can use a refinement of Jensen's inequality to estimate the right-hand side and deduce that, for any $\varepsilon_1>0$ there exists $\varepsilon_2>0$ such that
	\begin{equation}\label{estmon}
	\int_{\widetilde M_\tau} \left( f - \bar f \right)^2 d \nu \geq \varepsilon_1 \ \Longrightarrow \ \frac{d}{d\tau} \mathcal{E}_{\beta}(\widetilde{\Omega}_\tau) \leq -\varepsilon_2,
	\end{equation}
	where we have set
	$$\bar f=\int_{\widetilde M_\tau} f d\nu = \frac1{\omega_n} \int_{\widetilde M_\tau} \widetilde K^\beta d\mu.
	$$ 
	
	We want to use \eqref{estmon} to show that $f$ converges to a constant as $\tau \to -\infty$. To do this, we need some uniform control on the regularity of the solution; we begin by estimating $\bar f$. Since the area of $\widetilde M_\tau$ is bounded from both sides by convexity and the bounds on the radii, an upper bound on $\bar f$ holds in view of  \eqref{estgcf}. To find a lower bound, we argue as follows: if $\beta \geq 1$, we can use the H\"older inequality and the bound on $|\widetilde M_\tau|$ to conclude
	$$
	\bar f \geq \frac1{\omega_n} \left( \int_{\widetilde M_\tau} \widetilde K d\mu \right)^{\beta} |\widetilde M_\tau|^{1- \beta} = \left( \frac{|\widetilde M_\tau|}{\omega_n} \right)^{1- \beta} \geq \frac 1C,
	$$
	while if $\beta <1$, we can write
	$$
	\bar f \geq \frac1{\omega_n} \frac 1{(\sup \widetilde K)^{1-\beta}}\int_{\widetilde M_\tau}  \widetilde K d\mu = \frac 1{(\sup \widetilde K)^{1-\beta}} \geq \frac 1C.
	$$

	Observe that we lack a lower bound on $\widetilde{K}$, and that the methods of \cite{GN,AGN} to obtain such a bound do not seem to work in the backward limit $\tau \to -\infty$. This means that we do not know yet whether our problem remains uniformly parabolic as time decreases.
	We then follow a strategy introduced by Schulze in \cite{Schu} and later used in \cite{CRS,Sin}, which exploits the theory for degenerate or singular parabolic equations. In our case, we can do computations similar to \cite[\S 7.2]{CRS}, and find that the speed $K^\beta$ satisfies an equation of porous medium type, to which the H\"older regularity results from \cite{D,DF} can be applied. In this way, we find that there exist $\alpha \in (0,1)$ and $\eta>0$ such that, for any $x_0 \in M$ and $\tau_0 \in \mathbb{R}$, the parabolic $\alpha$-H\"older norm of $\widetilde K^\beta$ on $B_\eta(x_0) \times (\tau_0-\eta, \tau_0+\eta)$ is bounded by some $C$ independent of $(x_0, \tau_0)$.
	
	Now we are able to prove a lower bound on $\widetilde K$ as $\tau \to -\infty$. In fact, suppose that $\widetilde K(x_0,\tau_0)=\delta_0$ at some $(x_0,\tau_0)$, with $\delta_0>0$ suitably small depending on the constants $C$ of the previous estimates. Then, the bounds from below on $u_e$ and $\bar f$ imply that $|f(x_0,\tau_0) - \bar f(\tau_0)|$ is far from zero. By H\"older continuity, the same holds for $(x,\tau) \in B_\eta(x_0) \times (\tau_0-\eta, \tau_0+\eta)$. It follows that 
	$\int_{\widetilde M_{\tau}} \left( f - \bar f \right)^2 d \nu \geq \varepsilon_1$ for $\tau \in [\tau_0-\eta,\tau_0+\eta]$, where $\epsilon_1,\eta$ do not depend on $\tau_0$. In view of \eqref{estmon} and of the boundedness of the entropy, this can only occur on a finite number of intervals. We deduce that $\widetilde K(\cdot,\tau) > \delta_0$ for all $\tau<<0$.
	
	As we have shown that the rescaled Gaussian curvature is uniformly bounded from both sides, we deduce from the pinching condition that each principal curvature is bounded between two positive constants. This implies that the equation is uniformly parabolic with bounded coefficients, and we have uniform estimates on all the derivatives of the solution from Krylov-Safonov and Schauder theory. Standard arguments ensure precompactness for the family $\widetilde{M}_\tau$, so that every sequence $\widetilde{M}_{\tau_k}$ with $\tau_k \rightarrow -\infty$ as $k \rightarrow +\infty$ admits a subsequence converging in $C^\infty$ to a limit $\widetilde{M}_{-\infty}$.
	
	We now want to show that the right-hand side of \eqref{deriventro} must vanish on $\widetilde{M}_{-\infty}$. To this purpose, we need some continuity with respect to $\tau$ of the function $u_{e(\tau)}$. In Lemma 4.3 of \cite{AGN} it was proved that, for convex bodies $\Omega$ satisfying uniform bounds on the inner and outer radii, the entropy point is a continuous function of $\Omega$ with respect to the Hausdorff distance. Since the speed of our flow is bounded, we deduce that for any $\epsilon>0$ there exists $\eta>0$ such that
	$$
	||e(\tau)-e_{-\infty}|| \leq \epsilon, \qquad \forall \tau \in [\tau_k-\eta,\tau_k+\eta],
	$$
	for all $k$ sufficiently large, where $e_{-\infty}$ is the entropy point of $\widetilde{\Omega}_{-\infty}$. Therefore $u_{e(\tau)}$ is uniformly close to $u_{-\infty}$. By the regularity of $\widetilde K$, we find that the right hand side of \eqref{deriventro} is uniformly close to the same expression computed on $\widetilde{M}_{-\infty}$ for $\tau \in [\tau_k-\eta,\tau_k+\eta]$. Therefore, if the right-hand side of \eqref{deriventro} is nonzero on $\widetilde{M}_{-\infty}$, it is also uniformly negative for $\tau$ in a set of infinite measure, in contradiction with the boundedness of the entropy.
	
	We conclude that the right-hand side of \eqref{deriventro} vanishes on $\widetilde{M}_{-\infty}$. As already recalled, it is proved in \cite{BrendlDaskaChoi, DaskaChoi}, that the only convex hypersurface with this property is the sphere. This implies that the whole flow converges to a sphere as $\tau \to \infty$. On the other hand, it is known that the entropy attains its minimum value on the sphere among all convex bodies with fixed volume. By monotonicity, the entropy must remain constant on the flow $\widetilde{M}_\tau$, and $\widetilde{M}_\tau$ is a sphere for every $\tau$.
\end{proof}

\subsection{Flows with high degree of homogeneity}

Flows with general degree of homogeneity have in general more complicate analytic properties than in the $1$-homogeneous case. Various authors have proved convergence of convex hypersurfaces to a spherical profile for certain specific speeds with homogeneity $\alpha>1$, see \cite{AS,CRS,Chow1,Chow2,Schu}, under the requirement that the initial datum satisfies a suitably strong pinching condition. A general result of this form has been obtained in \cite{AMC}, where a large class of speeds is considered, with no structural assumptions such as convexity or concavity. Here we show that an ancient solution which satisfies the same pinching requirement as in \cite{AMC} is necessarily a shrinking sphere.

We consider a speed $f(\lambda)$ satisfying (H1)--(H2) for some $\alpha>1$ and its associated function $F(W)=f(\lambda(W))$.  For simplicity, we assume the normalization $f(1,\dots,1)=n$.  By the smoothness and homogeneity of $F$, there exists $\mu>0$ such that, for any matrices $A, B$:
\begin{equation}
| \ddot{F}^{kl,rs}(A)B_{kl}B_{rs}| \leq \mu H^{\alpha-2}|B|^2. \label{Fsecondo}
\end{equation}
Since by symmetry $\dot{F}|_{H I}=\alpha H^{\alpha - 1}I$, where $I$ is the identity matrix, it is easy to deduce
 \begin{align}
(\alpha H^{\alpha-1}-\mu H^{\alpha-2}|\mathring{h}|)I \leq &\dot{F}\leq (\alpha H^{\alpha-1}+\mu H^{\alpha-2}|\mathring{h}|)I,
\label{Fprimo}\\
H^{\alpha} -\frac{\mu}{2\alpha}H^{\alpha-2}|\mathring{h}|^2\leq &F \leq H^{\alpha}+\frac{\mu}{2\alpha}H^{\alpha-2} |\mathring{h}|^2.
\label{Fzero}
\end{align}
\begin{thm}
If $M_t=\varphi(M,t)$, $t \in (-\infty,0)$ is an ancient solution of \eqref{flowh}, with $F$ homogeneous of degree $\alpha > 1$ satisfying the additional conditions above and such that $|\mathring{h}|^2 \leq \epsilon H^2$ holds for all times $t \in (-\infty,0)$ for a suitable $\epsilon \in (0, \frac{1}{n(n-1)})$ depending only on $\mu, \alpha, n$, then it is a family of shrinking spheres.
\end{thm}
\begin{proof}
This time, similarly to the case of Mean Curvature Flow, we study the function
\begin{equation}
f_\sigma=\frac{|\mathring{h}|^2}{H^{2-\sigma}}
\end{equation}
and we want to show that it is identically zero for all negative times, for a suitable choice of $\sigma \in (0,1)$. We observe that, under our hypothesis, $f_0 \leq \epsilon$. If $\epsilon$ is small enough, we obtain from \eqref{Fzero}
\begin{equation}\label{below}
F \geq  \frac{\alpha}2 H^\alpha.
\end{equation}
The evolution equations for the relevant quantities are computed in [AMC]:
\begin{align}
\frac{\partial}{\partial t} H =& \mathcal{L}H+\ddot{F}^{kl,rs}\nabla^i h_{kl}\nabla_i h_{rs}+\dot{F}^{kl}h_{km}h^m_l H + (1-\alpha)F|h|^2\\[5 pt]
\frac{\partial}{\partial t} G =& \mathcal{L}G + [\dot{G}^{ij}\ddot{F}^{kl,rs}-\dot{F}^{ij}\ddot{G}^{kl,rs}]\nabla_i h_{kl}\nabla_j h_{rs}+\dot{F}^{kl}h_{km}h_l^m\dot{G}^{ij}h_{ij} \notag\\
&+ (1-\alpha)F\dot{G}^{ij}h_{im}h^m_j
\end{align}
if $G$ is any smooth symmetric function of the eigenvalues of the Weingarten operator.
Using these, we compute the evolution equation for $f_{\sigma}$:
\begin{align} \label{evfsigma}
\frac{\partial}{\partial t} f_{\sigma} &= \mathcal{L}f_{\sigma}+\frac{(2-\sigma)}{H}\dot{F}^{ij}\left[\nabla_i H \nabla_j f_{\sigma} + \nabla_j H \nabla_i f_{\sigma}\right]  \\[5pt]
&+ \frac{1}{H^{2-\sigma}}\left[2 \left(h^{ij} - \frac{1}{n}H g^{ij} -\frac{2-\sigma}{2} f_0 H g^{ij} \right)\ddot{F}^{kl,rs}\right]\nabla_i h_{kl} \nabla_j h_{rs} \notag\\[5pt]
& - \frac{(1-\sigma)(\sigma - 2)}{H^{2-\sigma}}\dot{F}^{ij}\nabla_i H \nabla_j H f_0 \notag \\[5pt]
&- \frac{2}{H^{2-\sigma}}\dot{F}^{ij}\left[\nabla_i h_{kl}\nabla_j h^{kl} -\frac{1}{n}\nabla_i H \nabla_j H\right] \notag \\[5pt]
&+ \sigma f_{\sigma}\dot{F}^{kl}h_{km}h^m_l + \frac{2(1-\alpha)}{n H^{2-\sigma}}F(nC-H|h|^2)-\frac{(1-\alpha)(2-\sigma)}{H}F|h|^2f_{\sigma},\notag
\end{align}
where $C=\sum_i \lambda_i^3$.
We estimate the terms in second row, using $f_0 \leq \epsilon$, $0<\sigma<1$ and the bound \eqref{Fsecondo}:
\begin{align}
&\left|2 \left(h^{ij} - \frac{1}{n}H g^{ij} -\frac{2-\sigma}{2} f_0 H g^{ij} \right)\ddot{F}^{kl,rs}\nabla_i h_{kl} \nabla_j h_{rs} \right|  \\[5pt]
&\leq 2 |\ddot{F}|\, |\nabla h|^2 H \sqrt{\epsilon}\left(1+\sqrt{n\epsilon} \right) \notag\\
&\leq 2 \mu H^{\alpha-1}|\nabla h|^2 \sqrt{\epsilon}\left(1+\sqrt{n\epsilon} \right).\notag
\end{align}
The term in the third row satisfies
\begin{align}
&\left| \dot{F}^{ij} (1-\sigma)(\sigma-2) f_0 \nabla_i H \nabla_j H\right| \leq 2 \epsilon |\dot{F}| |\nabla H|^2  \notag \\
&\leq 2 \epsilon \left(\alpha +\mu \sqrt{\epsilon}\right) H^{\alpha -1} \frac{n+2}{3} |\nabla h|^2.
\end{align}
The next term gives a negative contribution. In fact, it was shown in the proof of Theorem 5.1 in \cite{AMC} that
\begin{equation} \label{gradneg}
-2 \dot{F}^{ij}\left[\nabla_i h_{kl}\nabla_j h^{kl} -\frac{1}{n}\nabla_i H \nabla_j H\right]\leq -\frac{4(n-1)}{3n}H^{\alpha-1}(\alpha-\mu \sqrt{\epsilon})|\nabla h|^2.
\end{equation}
We observe that all positive terms occurring in the above estimates can be made arbitrarily small by choosing a small $\epsilon>0$, and the total contribution of the second, third and fourth row of \eqref{evfsigma} is nonpositive due to the negative term with the $\alpha$ factor in \eqref{gradneg}.

To estimate the terms in the last row, we rewrite them as
\begin{equation*}
\sigma f_\sigma \left(\dot{F}h_{km}h^m_l + \frac{(1-\alpha)}{H}F|h|^2\right) + \frac{2(1-\alpha)}{H}F\left(\frac{nC-H|h|^2}{nH^{1-\sigma}} - |h|^2f_{\sigma}\right).
\end{equation*}
The second part will give a negative contribution. To see this, we first apply Lemma 2.3 in \cite{AMC} to obtain
\begin{equation*}
nC-(1+n f_0)H|h|^2 \geq f_0(1+nf_0)(1-\sqrt{n(n-1)f_0})H^3 \geq \frac{1}{2}f_0H^3.
\end{equation*}
Then
\begin{equation}
\frac{nC-H|h|^2}{nH^{1-\sigma}}-|h|^2f_{\sigma} \geq \frac{nf_{\sigma} H |h|^2 + \frac{1}{2}f_{\sigma}H^3}{nH} - |h|^2 f_{\sigma} = f_{\sigma} \frac{H^2}{2n},
\end{equation}
and we conclude, using $\alpha>1$ and \eqref{below},
\begin{equation}
\frac{2(1-\alpha)}{H}F\left(\frac{nC-H|h|^2}{nH^{1-\sigma}} - F|h|^2f_{\sigma}\right) \leq (1-\alpha)f_{\sigma}\frac{H^{\alpha+1}}{2n}.
\end{equation}
At the same time, we have
\begin{align}
\sigma f_\sigma &\left(\dot{F}h_{km}h^m_l + \frac{(1-\alpha)}{H}F|h|^2\right)  \leq \sigma f_\sigma (\dot{F}h_{km}h^m_l)  \notag \\[5 pt]
&\leq \sigma f_{\sigma}\left(\alpha H^{\alpha-1}+\mu H^{\alpha -2}|\mathring{h}|\right)|h|^2 \leq \sigma f_{\sigma} (\alpha + \mu)H^{\alpha+1}.
\end{align}
The above estimates and \eqref{evfsigma} imply that, choosing $\sigma < \frac{\alpha - 1}{4n(\alpha+\mu)}$, we have
\begin{equation*}
\frac{\partial}{\partial t} f_{\sigma} \leq \mathcal{L}f_{\sigma}+\frac{2(2-\sigma)}{H}\dot{F}^{ij}\, \nabla_iH\nabla_jf_\sigma  + \frac{(1-\alpha)}{4n}f_{\sigma} H^{\alpha + 1}.
\end{equation*}
Since $f_{\sigma}\leq H^{\sigma}$, if we set $\psi(t)=\max_{M_t} f_\sigma$, we find
$$
\frac{d}{dt} \psi \leq - \frac{(\alpha-1)}{4n} \psi^{1+\frac{\alpha+1}{\sigma}}.
$$
An easy comparison argument, similar to the one in Proposition \ref{pro1},
shows that $\psi(t)$ cannot be defined for all $t \in (-\infty, 0)$ unless it is identically zero. Therefore, $|\mathring{h}|^2 \equiv 0$ on our solution and the assertion is proved.
\end{proof} \medskip

\noindent {\bf Acknowledgments} 
Carlo Sinestrari was partially supported by the research group GNAMPA of INdAM (Istituto Nazionale di Alta Matematica). \medskip

\bibliographystyle{abbrv}
\bibliography{biblio}
\end{document}